\documentclass[10pt]{article}
\font\smallit=cmti10
\font\smalltt=cmtt10

\usepackage{caption}
\usepackage{color}
\usepackage{amssymb,latexsym,amsmath,epsfig,amsthm} 
\usepackage{empheq}
\usepackage{here}
\usepackage{ascmac}
\makeatletter
\usepackage{here}
\usepackage{url}
\usepackage{comment}

\renewcommand\section{\@startsection {section}{1}{\z@}
{-30pt \@plus -1ex \@minus -.2ex}
{2.3ex \@plus.2ex}
{\normalfont\normalsize\bfseries\boldmath}}

\renewcommand\subsection{\@startsection{subsection}{2}{\z@}
{-3.25ex\@plus -1ex \@minus -.2ex}
{1.5ex \@plus .2ex}
{\normalfont\normalsize\bfseries\boldmath}}

\renewcommand{\@seccntformat}[1]{\csname the#1\endcsname. }

\makeatother
\newtheorem{theorem}{Theorem}
\newtheorem{lemma}{Lemma}

\theoremstyle{definition}
\newtheorem{definition}{Definition}[section]
\newtheorem{remark}{Remark}[section]
\newtheorem{problem}{Problem}
\newtheorem{sol}{Solution}
\newtheorem{example}{Example}[section]

\begin{document}

\begin{center}
\uppercase{\bf A Problem of Knot}
\vskip 20pt
{\bf Ryohei Miyadera }\\
{\smallit Keimei Gakuin, Kobe City, Japan}\\
{\tt runnerskg@gmail.com}
\vskip 10pt
{\bf Hikaru Manabe}\\
{\smallit Keimei Gakuin Kobe City, Japan}\\
{\tt urakihebanam@gmail.com}
\vskip 10pt
{\bf Aoi Murakami}\\
{\smallit Keimei Gakuin, Kobe City, Japan}\\
{\tt      }
\vskip 10pt
{\bf Shoma Morimoto}\\
{\smallit Keimei Gakuin, Kobe City, Japan}\\
{\tt      }


\end{center}
\vskip 20pt
\centerline{\smallit Received: , Revised: , Accepted: , Published: } 
\vskip 30pt


\centerline{\bf Abstract}
\noindent

In this article, the authors give the correct answer to the following problem, which is presented in the well-known problem book "CHALLENGING MATHEMATICAL PROBLEMS WITH ELEMENTARY SOLUTIONS"? by A. M. Yaglom and L. M. Yaglom.

There are six long blades of grass with the ends protruding above and below, and you will tie together the six upper ends in pairs and then tie together the six lower ends in pairs.
What is the probability that a ring will be formed when the blades of grass are tied at random in this fashion?

The solution in the above book needs to be corrected, and we will present a correct answer in this article. Therefore, we are the first persons to present a correct?answer to a problem in a book published in the USSR? in 1954.
By following the original idea of this problem book, we present the correct answer without using knowledge of higher knowledge, although we used a very basic knowledge of the Knot theory.

\pagestyle{myheadings} 
\markright{\smalltt   (  )\hfill} 
\thispagestyle{empty} 
\baselineskip=12.875pt 
\vskip 30pt

\section{Introduction}
In this article, the authors give an answer to the following Problem 1. This problem is presented in a well-known problem book "CHALLENGING MATHEMATICAL PROBLEMS WITH ELEMENTARY SOLUTIONS" \cite{problem}. by A. M. Yaglom and l. M. Yaglom. This problem is in page 25.

\begin{problem}\label{problem1}
In certain rural areas of Russia fortunes were once told in the following
way. A girl would hold six long blades of grass in her hand with the
ends protruding above and below; another girl would tie together the six
upper ends in pairs and then tie together the six lower ends in pairs.
If it turned out that the girl had thus tied the six blades of grass into a ring.
this was supposed to indicate that she would get married within a year.
What is the probability that a ring will be formed when the blades of
grass are tied at random in this fashion?
\end{problem}

In page 155 of \cite{problem}, there is the following solution, but this solution is not correct as you see later.
\begin{sol}
The upper ends of the six blades of grass can be joined in pairs in
$5 \times3 \times 1=15$  different ways (the first end can be tied to any of the other
five upper ends; then the first loose upper end can be tied to any of the
other three loose ends; then the two remaining loose ends must be tied
together). There are likewise 15 different ways of joining the lower ends.
Since the way the lower ends are joined is independent of the way the
upper ends are joined, there are a total of  $15 \times 15= 225$ equally likely
possible outcomes to the experiment.
Let us now compute the number of favorable outcomes. Let the
upper ends be connected in any of the 15 possible ways; let, say, the end of
the first blade be tied to the end of the second blade, the third to the fourth,
and the fifth to the sixth. In order for a ring be obtained, it is necessary
that the lower end of the first blade be tied to the lower end of the third,
fourth, fifth, or sixth blade; we thus have four possibilities for the lower
end of the first blade. Further, if the lower end of the first blade is joined
to that of the third blade, then the lower end of the second blade will have
to be joined to that of either the fifth or the sixth blade; here, we have only
two possibilities. After this is done we are left with only two loose ends,
which must be joined to each other. Combining all possibilities, we see
that for each of the 15 ways of joining the upper ends, there are exactly
$4 \times 2 = 8$ ways of joining the lower ends which lead to favorable outcomes
to the experiment. It follows from this that the total number of favorable
outcomes is $15 \times 8 = 120$.
Thus the probability to be computed is $15 \times 15 /(15 \times 15) = 8/15$.
\end{sol}

The probability that was calculated in the solution is for six tied blades to be connected, but
a connected tied six-blade is not always a circle. It can also be a knotted object.

In this article, we present an answer to Problem 1 by deciding in which case the tied blades will be more than one separated object, in which case it will be a circular object, and in which case it will be a knotted object.

As for the probabilities of getting a circle, a knotted object, a tied six-blade with more than one separated parts, 
it is up to the reader of this article to decide, since there are many different ways to decide which outcomes are equally probable.

\begin{remark}
Usually, the word ``ring'' does not include a knotted object, but even
If the word ``ring'' includes a circular object and a knotted object, it is meaningful to study the probability of getting a ring and the probability of getting a knotted object separately.
\end{remark}

\begin{figure}[H]
\begin{minipage}[t]{0.41\textwidth}
\begin{center}
\includegraphics[height=2cm]{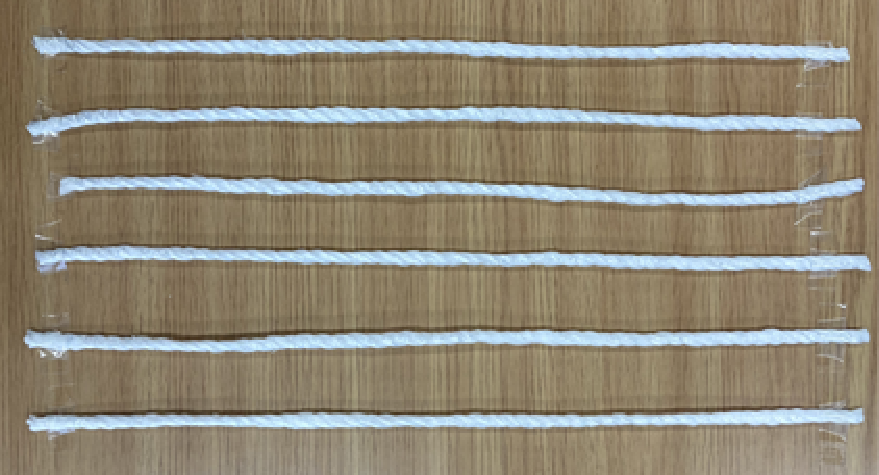}
\caption{six blades}
\label{blade6}
\end{center}
\end{minipage}
\begin{minipage}[t]{0.41\textwidth}
\begin{center}
\includegraphics[height=2cm]{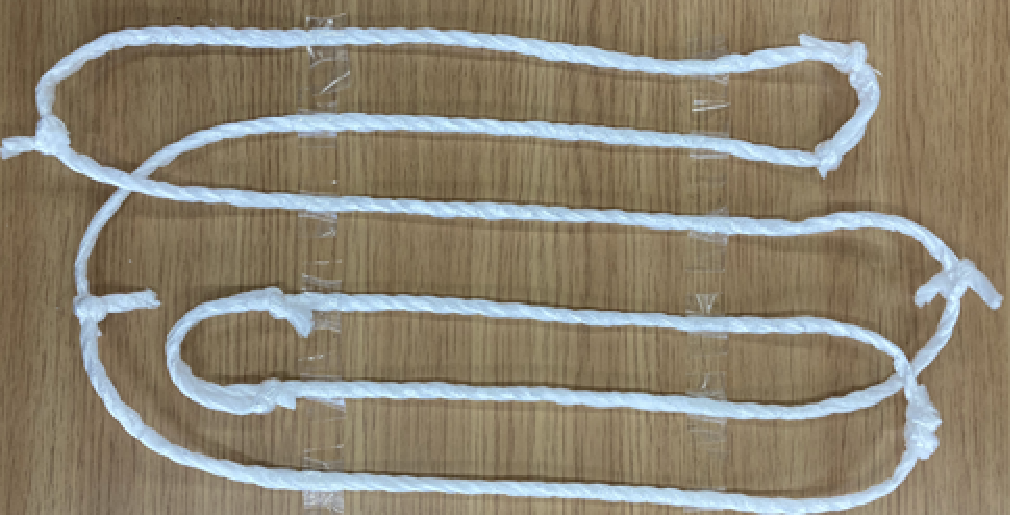}
\caption{tied six blades}
\label{tiedsixblades}
\end{center}
\end{minipage}
\end{figure}
We give numbers $1,2,3,4,5,6$ to six blades in
Figure \ref{blade6}. Then, we have the
graphic in Figure \ref{lines}.
Therefore, we have Blade 1, Blade 2, Blade 3, Blade 4, Blade 5, and Blade 6.
The tied six-blade in Figure \ref{tiedsixblades} can be represented by the 
graphics in Figure \ref{c4c1}.
The tied six-blade in Figure \ref{c4c1} consists of the tied left ends in 
Figure \ref{c4L} and the tied right ends in Figure \ref{c1R}.

\begin{figure}[H]
\begin{minipage}[t]{0.24\textwidth}
\begin{center}
\includegraphics[height=2.5cm]{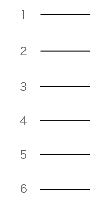}
\captionsetup{labelsep = period}
\caption{line numbers}
\label{lines}
\end{center}
\end{minipage}
\begin{minipage}[t]{0.24\textwidth}
\begin{center}
\includegraphics[height=2cm]{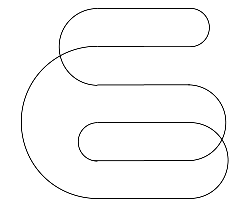}
\caption{c4c1}
\label{c4c1}
\end{center}
\end{minipage}
\begin{minipage}[t]{0.24\textwidth}
\begin{center}
\includegraphics[height=2cm]{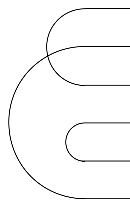}
\caption{c4}
\label{c4L}
\end{center}
\end{minipage}
\begin{minipage}[t]{0.24\textwidth}
\begin{center}
\includegraphics[height=2cm]{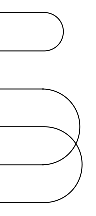}
\caption{c1}
\label{c1R}
\end{center}
\end{minipage}
\end{figure}

\begin{example}\label{crosscases}
There are four cases for the cross points in Figure \ref{c4c1}, and
The tied six-blade in Figure \ref{tiedsixblades} corresponds to Figure \ref{c4c1b}.

\begin{figure}[H]
\begin{minipage}[t]{0.23\textwidth}
\begin{center}
\includegraphics[height=2cm]{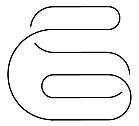}
\caption{Case 1}
\label{c4c1a}
\end{center}
\end{minipage}
\begin{minipage}[t]{0.23\textwidth}
\begin{center}
\includegraphics[height=2cm]{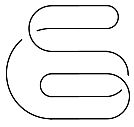}
\caption{Case 2}
\label{c4c1b}
\end{center}
\end{minipage}
\begin{minipage}[t]{0.23\textwidth}
\begin{center}
\includegraphics[height=2cm]{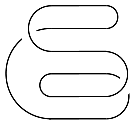}
\caption{Case 3}
\label{c4c1c}
\end{center}
\end{minipage}
\begin{minipage}[t]{0.23\textwidth}
\begin{center}
\includegraphics[height=2cm]{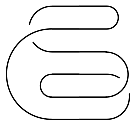}
\caption{Case 4}
\label{c4c1d}
\end{center}
\end{minipage}
\end{figure}
\end{example}

As we see in Example \ref{crosscases}, 
Figure \ref{c4c1}  does not present the relative positions of the two blades at the crossing points, so 
it is not specific enough to describe Figure \ref{c4c1b}, and we need Figure  \ref{tiedsixblades} to describe Figure \ref{c4c1b} precisely, however,
in most cases, we can determine that the tied six-blade will be a circle without knowing the relative positions of two blades at the crossing points.
For example, it is easy to see that the tied six-blade in Figure \ref{c4c1}  can be transformed into a circle by a Reidemeister move.

Therefore, we study figures without the relative positions of two blades at the crossing points until we get a certain number of six-blade
that we need to study the relative positions of two blades at the crossing points.

Next, we define 16 right-. We use the same names for the 16 left-end, too.
Each set of left end and right end has a name and a matrix.
For example, the right end in Figure \ref{a1} has the name $A1$ and a matrix $\left(
\begin{array}{ccc}
 1 & 3 & 5 \\
 2 & 4 & 6 \\
\end{array}
\right)$.
This matrix means that we tie Blade 1 and Blade 2, Blade 3 and Blade 4, and finally, Blade 5 and Blade 6.

\begin{definition}
We have four types of right-end $A1,A2,B1,B2$
\begin{figure}[H]
\begin{minipage}[t]{0.21\textwidth}
\begin{center}
\includegraphics[height=2.5cm]{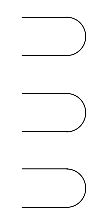}
\caption{A1}
\label{a1}
$\left(
\begin{array}{ccc}
 1 & 3 & 5 \\
 2 & 4 & 6 \\
\end{array}
\right)$
\end{center}
\end{minipage}
\begin{minipage}[t]{0.21\textwidth}
\begin{center}
\includegraphics[height=2.5cm]{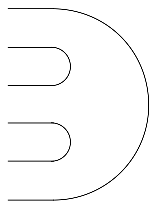}
\caption{A2}
\label{a2}
$\left(
\begin{array}{ccc}
 1 & 2 & 4 \\
 6 & 3 & 5 \\
\end{array}
\right)$
\end{center}
\end{minipage}
\begin{minipage}[t]{0.21\textwidth}
\begin{center}
\includegraphics[height=2.5cm]{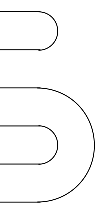}
\caption{B1}
\label{b1}
$\left(
\begin{array}{ccc}
 1 & 3 & 4 \\
 2 & 6 & 5 \\
\end{array}
\right)$
\end{center}
\end{minipage}
\begin{minipage}[t]{0.21\textwidth}
\begin{center}
\includegraphics[height=2.5cm]{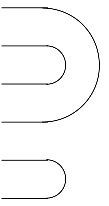}
\caption{B2}
\label{b2}
$\left(
\begin{array}{ccc}
 1 & 2 & 5 \\
 4 & 3 & 6 \\
\end{array}
\right)$
\end{center}
\end{minipage}
\end{figure}
We have four types of right-end $B3,C1,C2,C3$
\begin{figure}[H]
\begin{minipage}[t]{0.21\textwidth}
\begin{center}
\includegraphics[height=2.5cm]{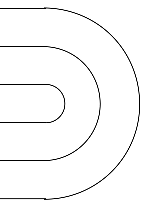}
\caption{B3}
\label{b3}
$\left(
\begin{array}{ccc}
 1 & 2 & 3 \\
 6 & 5 & 4 \\
\end{array}
\right)$
\end{center}
\end{minipage}
\begin{minipage}[t]{0.21\textwidth}
\begin{center}
\includegraphics[height=2.5cm]{C1.eps}
\caption{C1}
\label{c1}
$\left(
\begin{array}{ccc}
 1 & 3 & 4 \\
 2 & 5 & 6 \\
\end{array}
\right)$
\end{center}
\end{minipage}
\begin{minipage}[t]{0.21\textwidth}
\begin{center}
\includegraphics[height=2.5cm]{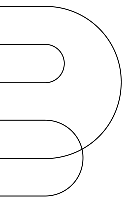}
\caption{C2}
\label{c2}
$\left(
\begin{array}{ccc}
 1 & 2 & 4 \\
 5 & 3 & 6 \\
\end{array}
\right)$
\end{center}
\end{minipage}
\begin{minipage}[t]{0.21\textwidth}
\begin{center}
\includegraphics[height=2.5cm]{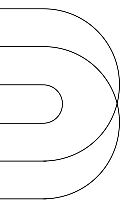}
\caption{C3}
\label{c3}
$\left(
\begin{array}{ccc}
 1 & 2 & 3 \\
 5 & 6 & 4 \\
\end{array}
\right)$
\end{center}
\end{minipage}
\end{figure}  
We have four types of right-end $C4,C5,C6,D1$
\begin{figure}[H]
\begin{minipage}[t]{0.21\textwidth}
\begin{center}
\includegraphics[height=2.5cm]{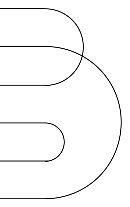}
\caption{C4}
\label{c4}
$\left(
\begin{array}{ccc}
 1 & 2 & 4 \\
 3 & 6 & 5 \\
\end{array}
\right)$
\end{center}
\end{minipage}
\begin{minipage}[t]{0.21\textwidth}
\begin{center}
\includegraphics[height=2.5cm]{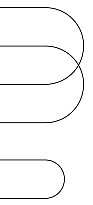}
\caption{C5}
\label{c5}
$\left(
\begin{array}{ccc}
 1 & 2 & 5 \\
 3 & 4 & 6 \\
\end{array}
\right)$
\end{center}
\end{minipage}
\begin{minipage}[t]{0.21\textwidth}
\begin{center}
\includegraphics[height=2.5cm]{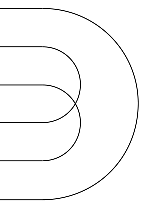}
\caption{C6}
\label{c6}
$\left(
\begin{array}{ccc}
 1 & 2 & 3 \\
 6 & 4 & 5 \\
\end{array}
\right)$
\end{center}
\end{minipage}
\begin{minipage}[t]{0.21\textwidth}
\begin{center}
\includegraphics[height=2.5cm]{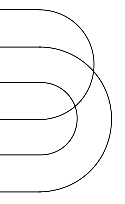}
\caption{D1}
\label{d1}
$\left(
\begin{array}{ccc}
 1 & 2 & 3 \\
 4 & 6 & 5 \\
\end{array}
\right)$
\end{center}
\end{minipage}
\end{figure}
We have four types of right-end $D2,D3,E1,E2$
\begin{figure}[H]
\begin{minipage}[t]{0.21\textwidth}
\begin{center}
\includegraphics[height=2.5cm]{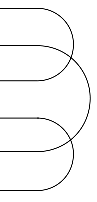}
\caption{D2}
\label{d2}
$\left(
\begin{array}{ccc}
 1 & 2 & 4 \\
 3 & 5 & 6 \\
\end{array}
\right)$
\end{center}
\end{minipage}
\begin{minipage}[t]{0.21\textwidth}
\begin{center}
\includegraphics[height=2.5cm]{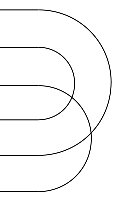}
\caption{D3}
\label{d3}
$\left(
\begin{array}{ccc}
 1 & 2 & 3 \\
 5 & 4 & 6 \\
\end{array}
\right)$
\end{center}
\end{minipage}
\begin{minipage}[t]{0.21\textwidth}
\begin{center}
\includegraphics[height=2.5cm]{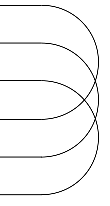}
\caption{E1}
\label{e1}
$\left(
\begin{array}{ccc}
 1 & 2 & 3 \\
 4 & 5 & 6 \\
\end{array}
\right)$
\end{center}
\end{minipage}
\begin{minipage}[t]{0.21\textwidth}
\begin{center}
\includegraphics[height=2.5cm]{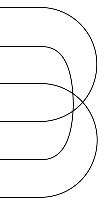}
\caption{E2}
\label{e2}
$\left(
\begin{array}{ccc}
 1 & 2 & 3 \\
 4 & 5 & 6 \\
\end{array}
\right)$
\end{center}
\end{minipage}
\end{figure}
\end{definition}

Throughout this article, we denote a  tied six-blade by the name of the left-end and the name of the right-end.

\begin{example}
The set of tied six blades in figures \ref{a1c5}, \ref{a1d1} and \ref{c3d3} are denoted by $A1C5$, $A1D1$ and 
$C3D3$.

\begin{figure}[H]
\begin{minipage}[t]{0.21\textwidth}
\begin{center}
\includegraphics[height=2cm]{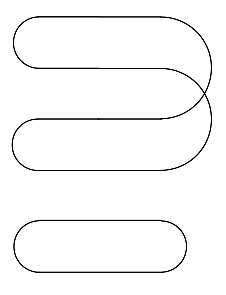}
\caption{A1C5}
\label{a1c5}
\end{center}
\end{minipage}
\begin{minipage}[t]{0.21\textwidth}
\begin{center}
\includegraphics[height=2cm]{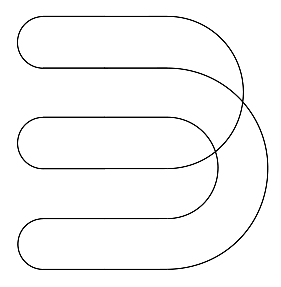}
\caption{A1D1}
\label{a1d1}
\end{center}
\end{minipage}
\begin{minipage}[t]{0.21\textwidth}
\begin{center}
\includegraphics[height=2cm]{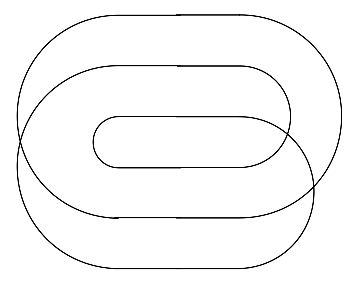}
\caption{C3D3}
\label{c3d3}
\end{center}
\end{minipage}
\end{figure}
\end{example}

\begin{definition}\label{defoftype}
$(i)$ A tied six-blade is said to be Type $I$ if
it consists of two separated parts.\\
$(ii)$  A tied six-blade is said to be Type $II$ if it does not consist of two separated parts.
\end{definition}

\begin{lemma}\label{samecolumn}
If the matrix of the left end and the matrix of the right end have the same column in common,
then the set of six blades is type $I$.
\end{lemma}
\begin{proof}
It is direct from Definition \ref{defoftype} and the definition of matrix attached to the left end and the right end.
\end{proof}

There are $16$ right ends  and $16$ left ends, so there are $256$ tied six-blades to be studied.
Next aim is to determine which of these $256$ tied six-blades are Type $I$ and which are Type $II$.

As you see in Theorem \ref{separate}, the tied six-blade $A1C1$  and $C1A1$  are Type $I$, but 
it is enough to state that $A1C1$ is Type $I$, because $A1C1$  is the left and the right mirror image of $C1A1$. 

In the remainder of this article, we often treat one of two tied six-blades that are mutually the left and the right mirror image of each other.

\begin{theorem}\label{separate}
Sets of six blades
$A1A1,A1B1,A1B2,A1B3,A1C1,A1C3,A1C5$\\
$A2A2,A2B1,A2B2,A2B3,A2C2,A2C4,A2C6$\\
$B1B1,B1C1,B1C4,B1D3,B1E1,B1E2$\\
$B2B2,B2C2,B2C5,B2D1,B2E1,B2E2$\\
$B3B3,B3C3,B3C6,B3E1,B3E2$\\
$C1C1,C1C2,C1C6,C1D1,C1D2$\\
$C2C2,C2C3,C2D2,C2D3$\\
$C3C3,C3C4,C3C5,C3D1,C3D3$\\
$C4C4,C4C5,C4D1,C4D2$\\
$C5C5,C5C6,C5C7,C5D2,C5D3$\\
$C6C6,C6D1,C6D3$\\
$D1D1,D1E1,D1E2$\\
$D2D2,D2E1,D2E2$\\
$D3D3,D3E1,D3E2$\\
$E1E1,E1E2$\\
$E2E2$
are type $I$.
\end{theorem}
\begin{proof}
This fact is direct from Lemma \ref{samecolumn} and matrices of each left and right end.
\end{proof}

The following Table shows which  six-blade is Type 1 and which is Type 2.  The symbol $N$ stands for "not connected", i.e., it is Type 1, and 
the symbol $C$ stands for "connected", i.e., it is Type 2.

\begin{figure}[H]
\begin{center}
\includegraphics[height=5cm]{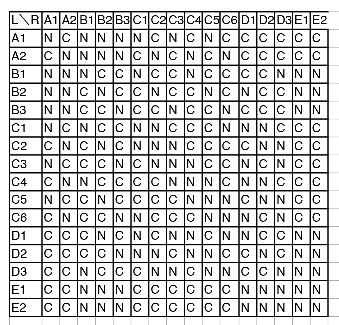}
\captionsetup{labelsep = period}
\caption{L-positions}
\label{leel}
\end{center}
\end{figure}
\begin{theorem}\label{cross12}
A knot that has less than or equal to two crossings is a trivial Knot, i.e., a circle.
\end{theorem}
\begin{proof}
This is a well-known fact. See \cite{knotatlas}.
\end{proof}

\begin{theorem}\label{trivial}
$(i)$ A1A2, A1C2, A1C4, A1C6, A1D1, A1D2, A1D3, A2C1, A2C3, A2C5, A2D1, A2D2, A2D3, B1B2, B1B3, B1C2, B1C3, B1C5, B1C6, B1D1, B1D2, B2B3, B2C1, B2C3, B2C4, B2C6, B2D2, B2D3 are trivial knots, i.e., circles.\\
$(ii)$ C1C3, C1C4, C1C5, C2C4, C2C5, C2C6,C3C5, C3C6,C4C6 are trivial knots.
\end{theorem}
\begin{proof}
$(i)$ $A1,A2,B1,B2,B3$ has no cross point, and each of $A2,B2,B3,C1,C2,C3,$ \\
$C4,C5,C6,D1,D2,D3$ has two cross points at most.
Therefore each of the tied six-blade  has two cross points at most, and hence by Theorem \ref{cross12},
they are trivial knots, i.e., circles.
$(i)$ Each of $C1,C2,C3,C4,C5,C6$ has one cross point, and hence each of the tied six-blade has two crossing, and hence by Theorem \ref{cross12},
they are trivial knots, i.e., circles.
\end{proof}

By Theorem \ref{separate} and Theorem \ref{trivial}, what remain to be studied are the following group of tied six-blades.
$A1E1, A1E2,A2E1, A2E2, C1D3, C1E1, C1E2, C2D1, C2E1$ \\
$, C2E2,C3D2, C3E1,C3E2,C4D3, C4E1, C4E2,C5D1, C5E1, C5E2,C6D2,$ \\
 $C6E1, C6E2,D1D2,D1D3,D2D3.$


\begin{figure}[H]
\begin{minipage}[t]{0.21\textwidth}
\begin{center}
\includegraphics[height=1.5cm]{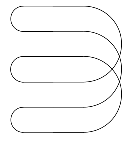}
\caption{A1E1}
\label{a1e1}
\end{center}
\end{minipage}
\begin{minipage}[t]{0.21\textwidth}
\begin{center}
\includegraphics[height=1.5cm]{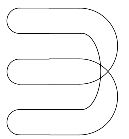}
\caption{A1E2}
\label{a1e2}
\end{center}
\end{minipage}
\begin{minipage}[t]{0.21\textwidth}
\begin{center}
\includegraphics[height=1.5cm]{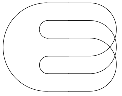}
\caption{A2E1}
\label{a2e1}
\end{center}
\end{minipage}
\begin{minipage}[t]{0.21\textwidth}
\begin{center}
\includegraphics[height=1.5cm]{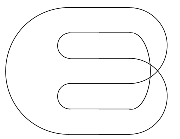}
\caption{A2E2}
\label{a2e2}
\end{center}
\end{minipage}
\end{figure}

\begin{figure}[H]
\begin{minipage}[t]{0.21\textwidth}
\begin{center}
\includegraphics[height=1.5cm]{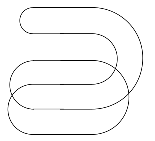}
\caption{C1D3}
\label{d3c1}
\end{center}
\end{minipage}
\begin{minipage}[t]{0.21\textwidth}
\begin{center}
\includegraphics[height=1.5cm]{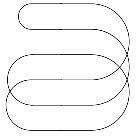}
\caption{C1E1}
\label{cie1}
\end{center}
\end{minipage}
\begin{minipage}[t]{0.21\textwidth}
\begin{center}
\includegraphics[height=1.5cm]{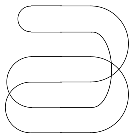}
\caption{C1E2}
\label{e2c1}
\end{center}
\end{minipage}
\begin{minipage}[t]{0.21\textwidth}
\begin{center}
\includegraphics[height=1.5cm]{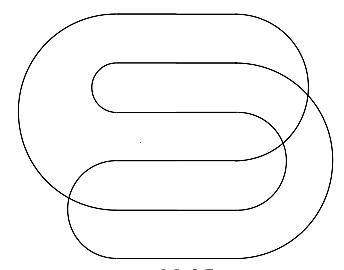}
\caption{C2D1}
\label{d1c2}
\end{center}
\end{minipage}
\end{figure}

\begin{figure}[H]
\begin{minipage}[t]{0.21\textwidth}
\begin{center}
\includegraphics[height=1.5cm]{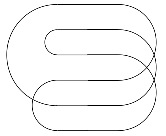}
\caption{C2E1}
\label{e1c2}
\end{center}
\end{minipage}
\begin{minipage}[t]{0.21\textwidth}
\begin{center}
\includegraphics[height=1.5cm]{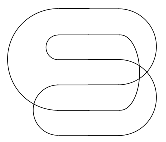}
\caption{C2E2}
\label{e2c2}
\end{center}
\end{minipage}
\begin{minipage}[t]{0.21\textwidth}
\begin{center}
\includegraphics[height=1.5cm]{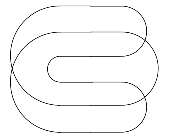}
\caption{C3D2}
\label{d2c3}
\end{center}
\end{minipage}
\begin{minipage}[t]{0.21\textwidth}
\begin{center}
\includegraphics[height=1.5cm]{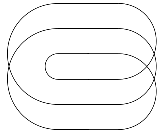}
\caption{C3E1}
\label{e1c3}
\end{center}
\end{minipage}
\end{figure}

\begin{figure}[H]

\begin{minipage}[t]{0.21\textwidth}
\begin{center}
\includegraphics[height=1.5cm]{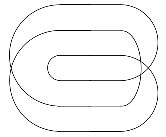}
\caption{C3E2}
\label{e2c3}
\end{center}
\end{minipage}
\begin{minipage}[t]{0.21\textwidth}
\begin{center}
\includegraphics[height=1.5cm]{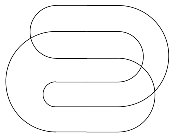}
\caption{C4D3}
\label{d3c4}
\end{center}
\end{minipage}
\begin{minipage}[t]{0.21\textwidth}
\begin{center}
\includegraphics[height=1.5cm]{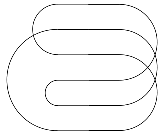}
\caption{C4E1}
\label{w1c4}
\end{center}
\end{minipage}
\begin{minipage}[t]{0.21\textwidth}
\begin{center}
\includegraphics[height=1.5cm]{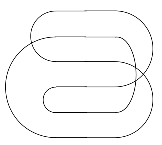}
\caption{C4E2}
\label{e2c4}
\end{center}
\end{minipage}
\end{figure}

\begin{figure}[H]
\begin{minipage}[t]{0.21\textwidth}
\begin{center}
\includegraphics[height=1.5cm]{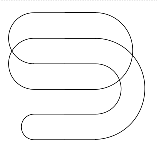}
\caption{C5D1}
\label{c5d1}
\end{center}
\end{minipage}
\begin{minipage}[t]{0.21\textwidth}
\begin{center}
\includegraphics[height=1.5cm]{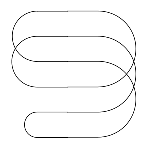}
\caption{C5E1}
\label{e1c5}
\end{center}
\end{minipage}
\begin{minipage}[t]{0.21\textwidth}
\begin{center}
\includegraphics[height=1.5cm]{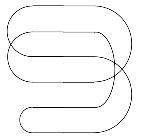}
\caption{C5E2}
\label{e2c5}
\end{center}
\end{minipage}
\begin{minipage}[t]{0.21\textwidth}
\begin{center}
\includegraphics[height=1.5cm]{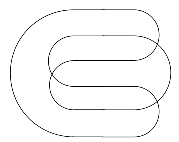}
\caption{C6D2}
\label{d2c6}
\end{center}
\end{minipage}
\end{figure}

\begin{figure}[H]
\begin{minipage}[t]{0.21\textwidth}
\begin{center}
\includegraphics[height=1.5cm]{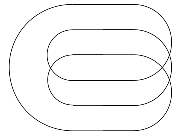}
\caption{C6E1}
\label{e1c6}
\end{center}
\end{minipage}
\begin{minipage}[t]{0.21\textwidth}
\begin{center}
\includegraphics[height=1.5cm]{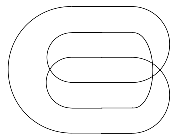}
\caption{C6E2}
\label{e2c6}
\end{center}
\end{minipage}
\begin{minipage}[t]{0.21\textwidth}
\begin{center}
\includegraphics[height=1.5cm]{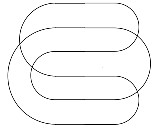}
\caption{D1D2}
\label{d1d2}
\end{center}
\end{minipage}
\begin{minipage}[t]{0.21\textwidth}
\begin{center}
\includegraphics[height=1.5cm]{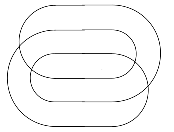}
\caption{D1D3}
\label{d1d3}
\end{center}
\end{minipage}
\end{figure}

\begin{figure}[H]
\begin{minipage}[t]{0.21\textwidth}
\begin{center}
\includegraphics[height=1.5cm]{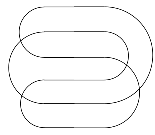}
\caption{D2D3}
\label{d3d2}
\end{center}
\end{minipage}
\end{figure}

Now we have the above 25 tied six-blades. Among them, there are pairs that are vertically symmetrical. 
\begin{lemma}
$(i)$ $C1E1$ is vertically symmetrical to $C5E1$.\\
$(ii)$  $C1D3$ is vertically symmetrical to $C5D1$.\\
$(iii)$  $C2D1$ is vertically symmetrical to $C4D3$.\\
$(iv)$  $D1D2$ is vertically symmetrical to $D2D3$.\\
$(v)$  $C2E1$ is vertically symmetrical to $C4E1$.\\
$(vi)$  $C1E2$ is vertically symmetrical to $C5E2$.\\
$(v)$  $C2E2$ is vertically symmetrical to $C4E2$.
\end{lemma}
\begin{proof}
These are direct from shapes of tied six-blades.
\end{proof}

\begin{lemma}
$(i)$ $A1E1$ is a trivial knot, i.e., a circle.\\
$(ii)$  $A2E2$ is a trivial knot.
\end{lemma}
\begin{proof}
By a Reidemeister move, we can reduce the number of crossings of $A1E1$ and $A2E2$ less than three, and hence by Lemma \ref{cross12},
these are trivial knots, i.e., circles.
\end{proof}

\begin{figure}[H]
\begin{minipage}[t]{0.23\textwidth}
\begin{center}
\includegraphics[height=2cm]{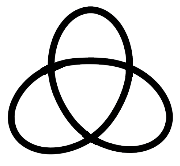}
\caption{three cross points}
\label{threecross}
\end{center}
\end{minipage}
\begin{minipage}[t]{0.23\textwidth}
\begin{center}
\includegraphics[height=2cm]{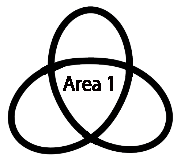}
\caption{The central area}
\label{central}
\end{center}
\end{minipage}
\begin{minipage}[t]{0.45\textwidth}
\begin{center}
\includegraphics[height=2cm]{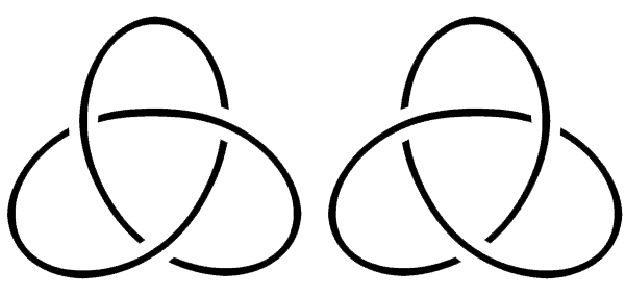}
\caption{trefoil}
\label{trefoil}
\end{center}
\end{minipage}
\end{figure}

\begin{lemma}\label{lemmafortrefoil}
There are 8 cases of relative positions of crossings in Figure \ref{threecross}, and two of them are Knots in Figure \ref{trefoil}, and the other six cases are trivial knots, i.e., circles.
\end{lemma}
\begin{proof}
We can prove this by illustrating eight cases with easy so we omit the proof.
\end{proof}

\begin{lemma}
$A1E2$, $A2E1$, $C1D3$, $C2D1$, $C3D2$, $C6D2$ are mathematically the same as the figure of Figure \ref{threecross}. Hence there are six cases that the set of six blades becomes a trivial knot, and two cases that the set of six blades becomes a trefoil.
\end{lemma}
\begin{proof}
To prove that a tied six-blade is mathematically the same as the figure of Figure \ref{threecross}, we have only to find an area in a tied six-blade that correspond the area 1 of the figure in Figure \ref{central}
.  This area has three points in its boundary.
It is easy to find such an area in $A1E2$, $C5D1$, $C3D2$. For the cases of  $A2E1$,  $C2D1$, $C6D2$, it is not difficult to find an area with  three points in its boundary, and by
moving an arc on the left side to the right you can show that these are mathematically the same as the figure of Figure \ref{threecross}.
Therefore, by Lemma \ref{lemmafortrefoil}, we can finish the proof.
\end{proof}

\begin{lemma}
$(i)$ By a Reidemeister move, $C1E1$ becomes $C1D3$.\\
$(ii)$ By a Reidemeister move, $C3E1$ becomes $C3D2$.\\
$(iii)$ By a Reidemeister move, $C2E2$ becomes $C2D1$.\\
$(ii)$ By a Reidemeister move, $C6E2$ becomes $C6D2$.\\

\end{lemma}
\begin{proof}
$(i)$ By a Reidemeister move,  we can eliminate the crossing that is  on the upper-right of the figure. Then we get $C1D3$.\\
$(ii)$ By a Reidemeister move, we can eliminate the crossing that is in the middle right of the figure. Then we get $C3D2$.\\
$(iii)$ By a Reidemeister move, we can eliminate the crossing that is in the upper-right of the figure. Then we get $C2D1$.\\
$(iii)$ By a Reidemeister move, we can eliminate the crossing that is in the middle right of the figure.  Here we move the arc on the left to the right side.
 Then we get $C6D2$.
\end{proof}

\begin{figure}[H]
\begin{minipage}[t]{0.24\textwidth}
\begin{center}
\includegraphics[height=2.4cm]{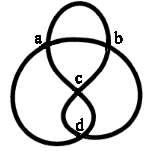}
\caption{eightfigure}
\label{eightfigure}
\end{center}
\end{minipage}
\begin{minipage}[t]{0.24\textwidth}
\begin{center}
\includegraphics[height=2.4cm]{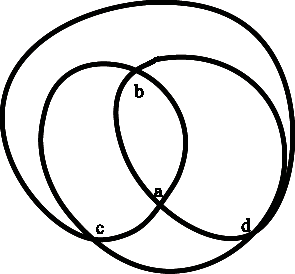}
\caption{eightfigure2}
\label{eightfigure2}
\end{center}
\end{minipage}
\begin{minipage}[t]{0.24\textwidth}
\begin{center}
\includegraphics[height=3.5cm]{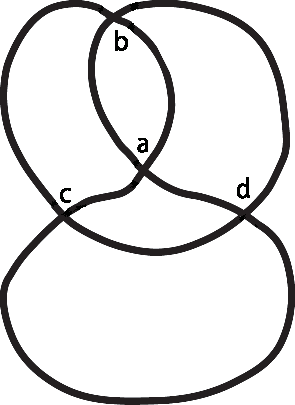}
\caption{eightfigurebb}
\label{eightfigurebb}
\end{center}
\end{minipage}
\begin{minipage}[t]{0.24\textwidth}
\begin{center}
\includegraphics[height=2.4cm]{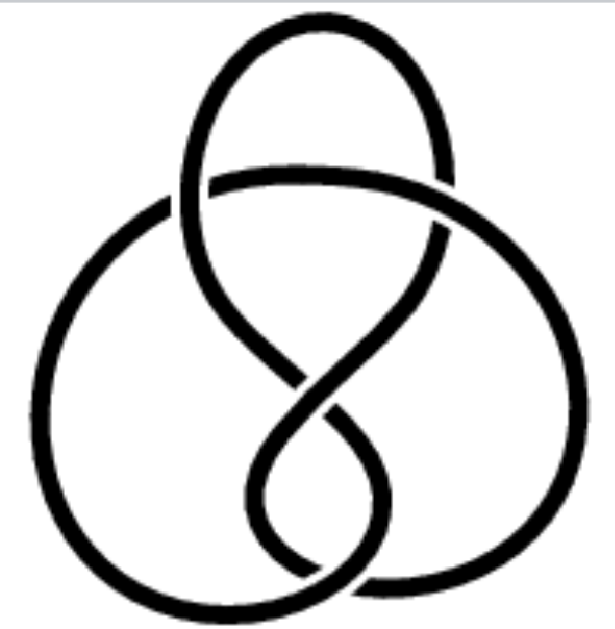}
\caption{eightknot}
\label{eightknot}
\end{center}
\end{minipage}
\end{figure}

\begin{lemma}\label{lemmaforeightfig}
There are 16 cases of relative positions of crossings in Figure \ref{eightfigure}, and two of them are Knots in Figure \ref{trefoil}, and the other sixteen cases are trivial knots, i.e., circles.
\end{lemma}
\begin{proof}
We can prove this by illustrating sixteen cases with easy so we omit the proof.
\end{proof}

\begin{theorem}
$(i)$ $D1D2$,  $E1C2$,  $E2C1$ are mathematically the same as
the figure in $(\ref{eightfigure})$. Hence, for each of them, we have 14 cases in the set of six blades becoming trivial knot, and 2 cases in the set of six blades becoming a Figure-eight knot.\\
$(ii)$ $D1D3$ , $E1C6$ , $E2C3$ are mathematically the same as
the figure in $(\ref{eightfigure})$. Hence, for each of them, we have 14 cases in the set of six blades becoming trivial knot, and 2 cases in the set of six blades becoming a Figure-eight knot.
\end{theorem}
\begin{proof}
$(i)$ For $D1D2$,  $E1C2$,  $E2C1$, please look at points $a,b,c,e$ in figures \ref{d1d2b}, \ref{e1c2b} and \ref{e2c1b}, and compare them to Figure \ref{eightfigure}.\\
$(ii)$ For $D1D3$ , $E1C6$ , $E2C3$, please look at points $a,b,c,e$ in figures 
\ref{d1d2b}, \ref{e1c2b} and \ref{e2c1b}, and compare them to Figure \ref{eightfigure2}.\\
By moving the arc on the upper side to lower side, we get Figure \ref{eightfigurebb} from Figure \ref{eightfigure2}. It is clear that Figure \ref{eightfigure2} is mathematically the same as 
Figure $(\ref{eightfigure})$.
Therefore, by Lemma \ref{lemmaforeightfig}, we finish the proof.
\end{proof}

\begin{figure}[H]
\begin{minipage}[t]{0.31\textwidth}
\begin{center}
\includegraphics[height=2.5cm]{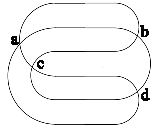}
\caption{D1D2}
\label{d1d2b}
\end{center}
\end{minipage}
\begin{minipage}[t]{0.31\textwidth}
\begin{center}
\includegraphics[height=2.5cm]{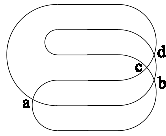}
\caption{C2E1}
\label{e1c2b}
\end{center}
\end{minipage}
\begin{minipage}[t]{0.31\textwidth}
\begin{center}
\includegraphics[height=2.5cm]{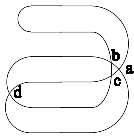}
\caption{C1E2}
\label{e2c1b}
\end{center}
\end{minipage}
\end{figure}

\begin{figure}[H]
\begin{minipage}[t]{0.31\textwidth}
\begin{center}
\includegraphics[height=2.5cm]{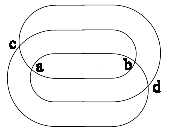}
\caption{D1D3}
\label{d1d3b}
\end{center}
\end{minipage}
\begin{minipage}[t]{0.31\textwidth}
\begin{center}
\includegraphics[height=2.5cm]{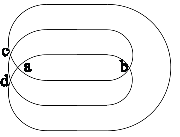}
\caption{C6E1}
\label{e1c6b}
\end{center}
\end{minipage}
\begin{minipage}[t]{0.31\textwidth}
\begin{center}
\includegraphics[height=2.5cm]{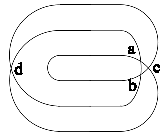}
\caption{C3E2}
\label{e2c3b}
\end{center}
\end{minipage}
\end{figure}

By the lemmas and theorems above, we studied all the cases and learned which tied six-blade is a circle, which is a non-trivial knot.

\bibliographystyle{amsplain}

\end{document}